\documentclass[12pt]{article}

\usepackage[utf8]{inputenc}
\usepackage[T2A]{fontenc}
\usepackage[english]{babel}
\usepackage[title]{appendix}

\usepackage{amssymb}
\usepackage{amsmath}
\usepackage{amsthm} 
\usepackage{mathtext}
\usepackage{amsfonts}
\usepackage{graphicx,geometry}
\usepackage[title]{appendix}
\usepackage{cite}
\usepackage{xcolor}

\newcommand{\refpar}{Sect.~}

\newcommand{\Do}{{\raisebox{0.2ex}{\(\stackrel{\circ}{\mathcal{D}}\)}}{}}

\DeclareSymbolFont{myletters}{OML}{ztmcm}{m}{it}
\DeclareMathSymbol{\uplambda}{\mathord}{myletters}{"15}

\newtheorem{theorem}{Theorem}
\newtheorem{lemma}{Lemma}
\newtheorem{proposition}{Proposition}

\theoremstyle{definition}

\newtheorem{remark}{Remark}

\begin{document} 
\title{Non-radial minimizers for Hardy--Sobolev \\ inequalities in non-convex cones}
\author{A.~I.~Nazarov\footnote{e-mail: al.il.nazarov@gmail.com}, 
\setcounter{footnote}{6}
N.~V.~Rastegaev\footnote{e-mail: rastmusician@gmail.com} \\
\small{St. Petersburg Department of Steklov Mathematical Institute} \\
\small{of Russian Academy of Sciences} \\ \small{
27 Fontanka, 191023, St. Petersburg, Russia} }
\renewcommand{\today}{}
\maketitle
\abstract{
The symmetry breaking is obtained for Neumann problems driven by $p$-Laplacian in certain non-convex cones. These problems are generated by the Hardy--Sobolev inequalities. In the case of the Sobolev inequality for the ordinary Laplacian this problem was investigated in \cite{Ciraolo2024}. 

Such problems have obvious radial solutions --- Talenti--Bliss type functions of $|x|$. However, under a certain restriction on the first Neumann eigenvalue $\lambda_1(D)$ of the Beltrami--Laplace operator on the spherical cross-section $D$ of the cone we prove this radial solution cannot be an extremal function, therefore minimizer must be non-radial. This leads to multiple solutions for the corresponding Neumann problem. 
}

\section{Introduction}

Let $n\geqslant 2$, $1<p<n$, $0 < \sigma\leqslant 1$. Denote $q = p_\sigma^* = \frac{np}{n-\sigma p}$. The Hardy--Sobolev (Sobolev for $\sigma=1$) inequality in the whole space reads as follows:
\begin{equation}\label{eq:HardiSobolevSpace}
\|r^{\sigma-1}v\|_{q, \mathbb{R}^n} \leqslant C \|\nabla v\|_{p, \mathbb{R}^n}, \quad v\in \mathcal{C}_0^\infty(\mathbb{R}^n).
\end{equation}
Hereinafter $r=|x|$. This allows us to define the space of functions
\[
\mathcal{D}_{p}(\mathbb{R}^n) = \{ u\in L_q(\mathbb{R}^n, r^{(\sigma-1)q}) : \nabla u\in L_p(\mathbb{R}^n) \}
\]
(notice that this space does not depend on $\sigma$) and extend the inequality \eqref{eq:HardiSobolevSpace} to this space. Moreover, the exact constant in this inequality is attained at radially symmetric Talenti--Bliss type functions
\begin{equation}\label{eq:w_def}
w(r) = \Big(1 + r^{\frac{\sigma (n-p) q}{n(p-1)}}\Big)^{-\frac{n}{\sigma q}}
\end{equation}
or their homotheties (and shifts in the case $\sigma=1$). This result was obtained in \cite{Aubin1976}, \cite{Talenti1976} independently for $\sigma=1$, in \cite{Lieb1983} for $0<\sigma<1$, $p=2$ and in \cite{Ghoussoub2000} for $0<\sigma<1$, $1<p<n$.

Consider a domain $D$ with strictly Lipschitz boundary on the unit sphere $\mathbb{S}^{n-1}$, and the cone $$\Sigma_D = \{xt: x\in D, t\in (0,+\infty)\} \subset \mathbb{R}^n$$ spanning $D$. Define the space
\[
\mathcal{D}_{p}(\Sigma_D) = \{ u\in L_q(\Sigma_D, r^{(\sigma-1)q}) : \nabla u\in L_p(\Sigma_D) \}.
\]
From \eqref{eq:HardiSobolevSpace} it follows that in the cone $\Sigma_D$ the inequality
\begin{equation}\label{eq:HardiSobolevCone}
\|r^{\sigma-1}v\|_{q, \Sigma_D} \leqslant C \|\nabla v\|_{p, \Sigma_D}, \quad v \in \mathcal{D}_{p}(\Sigma_D),
\end{equation}
holds. If the exact constant in \eqref{eq:HardiSobolevCone} is attained, then the necessary condition for a minimizer is the Neumann problem
\begin{equation}\label{eq:Neumann_problem}
\begin{cases} 
-\Delta_p u = \lambda \dfrac{u^{q-1}}{|x|^{(1-\sigma)q}} &\text{in } \Sigma_D, \\
\dfrac{\partial u}{\partial \nu} = 0 &\text{on } \partial \Sigma_D, \\
u > 0 &\text{in } \Sigma_D.
\end{cases}
\end{equation}
Here $\Delta_p u = \mathrm{div}(|\nabla u|^{p-2}\nabla u)$ is the so-called $p$-Laplacian and $\lambda$ is the Lagrange multiplier. 

This problem has an evident radial solution \eqref{eq:w_def}.
But the question whether the exact constant in \eqref{eq:HardiSobolevCone} is attained, and whether function \eqref{eq:w_def} is a minimizer or not, is non-trivial. The cases of symmetry breaking are therefore of interest.

In \cite{Nazarov2006} the attainability of the exact constant was proved for the Dirichlet case
\[
\|r^{\sigma-1}v\|_{q, \Sigma_D} \leqslant C \|\nabla v\|_{p, \Sigma_D}, \quad v \in \Do_{p}(\Sigma_D),
\]
for $1<p<\infty$, $0<\sigma<\min(1, n/p)$ (in the case $p=2$ this was previously proved in \cite{Egnell1992}). Note that the difference from \eqref{eq:HardiSobolevCone} is that functions $v$ vanish at the boundary of the cone. We mention this result, because it's translation to the Neumann case \eqref{eq:HardiSobolevCone} for $1<p<n$, $0<\sigma<1$ is very straightforward. 

In \cite{Lions1988} (see also \cite{Ciraolo2020}) convex cones were considered for $n\geqslant 2$, $1<p<n$, $\sigma=1$, and the exact constant in \eqref{eq:HardiSobolevCone} was shown to be attained only at radially symmetrical functions. 

In \cite{ClappPacella2019} the corresponding critical Neumann problem was considered for $n\geqslant 3$, $p=2$, $\sigma=1$. The attainability of exact constant in \eqref{eq:HardiSobolevCone} was established for $D$ that have a point with local convexity condition and the existence of non-radial minimizers was established in some cases. In \cite{Ciraolo2024}, a wider class of domains $D$ was introduced for which the minimizer is non-radial. 

In this paper we aim to provide a similar class of domains for the general case $n\geqslant 2$, $1<p<n$, $0<\sigma\leqslant 1$. For $0<\sigma<1$ the attainability of the exact constant is obtained similar to \cite{Nazarov2006} with no changes to the proof. For $\sigma=1$ we prove the attainability under certain restrictions on $D$. Then, for $0<\sigma\leqslant 1$ we show that the exact constant cannot be attained on radial functions if the first Neumann eigenvalue $\lambda_1(D)$ of the Beltrami--Laplace operator on $D$ is sufficiently small.

 Below $B(x, r)$ stands for the ball with center at $x$ and radius $r$. We denote by $\mathbb{R}^n_+$ the half-space and by $\mathbb{S}^{n-1}_+$ the half-sphere.
 Letter $C$ stands for various positive constants the values of which are not essential. 

The paper has the following structure. \refpar 2 contains the proof of attainability for the subcritical case $0<\sigma<1$ and (under certain restrictions on the domain $D$) for the Sobolev case $\sigma=1$. In \refpar 3 we prove the symmetry breaking for sufficiently small $\lambda_1(D)$. The Appendix contains the proof of an auxiliary concentration-compactness type statement applied in \refpar 2.

\section{Attainability of the exact constant}
\subsection{Subcritical case $0<\sigma<1$}

In this subsection we consider $0 < \sigma < 1$, $q = p_\sigma^* = \frac{np}{n-\sigma p} < p^* = \frac{np}{n-p}$. This case is, therefore, subcritical.

Define the Hardy--Sobolev quotient
\begin{equation}\label{eq:functional_Q_Hardy_Sobolev}
Q_{\Sigma_D}^\sigma(u) = \dfrac{\|\nabla u\|_{L_p(\Sigma_D)}^p}{\|r^{(\sigma-1)}u\|_{L_q(\Sigma_D)}^p}, \qquad u\in \mathcal{D}_{p}(\Sigma_D), \quad u\neq 0.
\end{equation}

We aim to demonstrate the attainability of
\[
S^\sigma_{\Sigma_D} = \inf\limits_{u\in\mathcal{D}_{p}(\Sigma_D)\setminus\{0\}} Q^\sigma_{\Sigma_D}(u).
\]

Consider a minimizing sequence $u_k$ for the functional \eqref{eq:functional_Q_Hardy_Sobolev}. Without loss of generality we can assume that all $u_k$ have bounded supports, otherwise we can cut them off at sufficiently large radii. Therefore, we can consider them as functions on a compactified cone $\overline \Sigma_D = \Sigma_{\overline{D}}\cup\{\infty\}$ (obtained by adding a single ``point at infinity'', which we will denote by $\infty$). In such space the following variant of Lemma 3.1 from \cite{Nazarov2006} can be formulated (see also \cite{Lions1984}).

\begin{proposition}
\label{proposition:lemma_3_1_Nazarov}
Let $u_k$ be a bounded sequence in $\mathcal{D}_{p}(\Sigma_D)$ with bounded supports. Then, up to a subsequence, $u_k\rightharpoondown u$, and the following weak convergences hold in the space of measures in $\overline \Sigma_D$:
\begin{align*}
|r^{\sigma-1}u_k|^q \rightharpoondown &\;\mu = |r^{\sigma-1}u|^q + \mu_0 \delta_0 + \mu_\infty \delta_\infty, \\
|\nabla u_k|^p \rightharpoondown &\;\widetilde{\mu} \geqslant |\nabla u|^p + S^\sigma_{\Sigma_D}\mu_0^{p/q}\delta_0 + S^\sigma_{\Sigma_D}\mu_\infty^{p/q}\delta_\infty,
\end{align*}
where $\mu_0, \mu_\infty \geqslant 0$.
\end{proposition}
From this proposition we obtain the attainability similar to Theorem 3.1 in \cite{Nazarov2006}.
\begin{theorem}
\label{Th1_Nazarov}
For $0<\sigma<1$ the exact constant $(S^\sigma_{\Sigma_D})^{1/p}$ in the inequality \eqref{eq:HardiSobolevCone} is attained.
\end{theorem}
We omit the proofs, since they do not differ from the proofs given in \cite{Nazarov2006} in any significant way.

\subsection{Sobolev case $\sigma=1$}

In this subsection we assume $\sigma=1$, $q=p^*=\frac{np}{n-p}$. Define the Sobolev quotient
\begin{equation}\label{eq:functional_Q}
Q_{\Sigma_D}(u) \equiv Q^1_{\Sigma_D}(u) = \dfrac{\|\nabla u\|_{L_p(\Sigma_D)}^p}{\|u\|_{L_q(\Sigma_D)}^p}, \qquad u\in \mathcal{D}_{p}(\Sigma_D), \quad u\neq 0.
\end{equation}
We aim to study the attainability of
\[
S_{\Sigma_D} \equiv S_{\Sigma_D}^1 =\inf\limits_{u\in\mathcal{D}_{p}(\Sigma_D)\setminus\{0\}} Q_{\Sigma_D}(u).
\]

If $D$ has a smaller area than a half-sphere, $|D| < |\mathbb{S}^{n-1}_+|$,
then by substituting the radial function \eqref{eq:w_def} with $\sigma=1$ into \eqref{eq:functional_Q} we obtain
\begin{equation*}
S_{\Sigma_D} \leqslant Q_{\Sigma_D}(w) = \left(\dfrac{|\Sigma_D|}{|\mathbb{S}^{n-1}|}\right)^{1-p/q} Q_{\mathbb{R}^n}(w) < 2^{-\frac{p}{n}} S_{\mathbb{R}^n} = S_{\mathbb{R}^n_+}.
\end{equation*}
In particular, this is clear for $D \subsetneq \mathbb{S}^{n-1}_+$. 

In some cases the inequality
\begin{equation}\label{eq:S_less_than_half}
S_{\Sigma_D} < S_{\mathbb{R}^n_+}
\end{equation}
holds for larger cones as well. The following lemma follows from the estimates and techniques provided in \cite{Demyanov2006}.
\begin{lemma}\label{lemma1}
If $p\leqslant \frac{n+1}{2}$, the boundary $\partial D$ is $\mathcal{C}^2$-smooth  and has a point with positive mean curvature, then \eqref{eq:S_less_than_half} holds.
\end{lemma}
\begin{proof}
All the necessary estimates are given in formulas (22), (23), (53) of the paper \cite{Demyanov2006}. Follow the scheme of \cite[Theorem 3.1]{Demyanov2006} to derive the assertion of this lemma from those estimates.
\end{proof}

\begin{remark}
For $p=2$ (and $n\geqslant 3$) any $D$ with $\partial D\in \mathcal{C}^2$ and $|D| < |\mathbb{S}^{n-1}_+|$ is covered by the result of Lemma \ref{lemma1}. But for general $1<p<n$ the condition $|D| < |\mathbb{S}^{n-1}_+|$ and Lemma \ref{lemma1} give different classes of domains satisfying \eqref{eq:S_less_than_half}.

Note also, that in \cite{ClappPacella2019, Ciraolo2024} the class of domains $D$ providing the inequality \eqref{eq:S_less_than_half} (see \cite[Definition 2.6]{ClappPacella2019}) is slightly more restricted.
\end{remark}

The main result of this subsection is formulated in the following theorem.

\begin{theorem}\label{theorem1}
If the boundary $\partial D \in \mathcal{C}^1$, and \eqref{eq:S_less_than_half} holds, then $S_{\Sigma_D}$ is attained.
\end{theorem}
\begin{proof}

We first need the standard auxiliary concentration--compactness statement similar to Proposition \ref{proposition:lemma_3_1_Nazarov}. The corresponding Proposition \ref{prop:lemma_2.2} is proved in the Appendix.

Consider a minimizing sequence $u_k$ for the functional \eqref{eq:functional_Q}. Similar to the previous subsection, we consider them functions on a compactified cone $\overline \Sigma_D$. Let's make certain transformations with them. Since $Q_{\Sigma_D}(a u(b x)) = Q_{\Sigma_D}(u(x))$ for all $a, b > 0$, we can assume without loss of generality that $u_k$ are normalized and keep exactly half of their mass in the unit ball:
\begin{equation*}\label{eq:u_k_normalize}
\|u_k\|_{L_q(\Sigma_D)} = 1, \qquad
\|u_k\|_{L_q(B(0,1)\cap\Sigma_D)} = \frac{1}{2}.
\end{equation*}



Without loss of generality $u_k$ satisfy the assumptions of Proposition \ref{prop:lemma_2.2}, therefore, using the notations of Proposition \ref{prop:lemma_2.2}, we have
\begin{align*}
\left( \,\int\limits_{\overline \Sigma_D} |u|^q\, dx \right)^{p/q} &+ \sum_{j \in J} \mu_j^{p/q} \leqslant \dfrac{1}{S_{\Sigma_D}} \left(\, \int\limits_{\overline \Sigma_D} |\nabla u|^p\, dx + \sum_{j \in J} S_{\Sigma_D} \mu_j^{p/q} \right) \\
&{} \leqslant \dfrac{1}{S_{\Sigma_D}} \int\limits_{\overline \Sigma_D} d\tilde{\mu} = \left(\,\int\limits_{\overline \Sigma_D} d\mu\right)^{p/q} = \left(\,\int\limits_{\overline \Sigma_D} |u|^q \, dx + \sum_{j \in J} \mu_j\right)^{p/q}.
\end{align*}
Since $\frac{p}{q} < 1$, only one of the terms could be non-zero. If it's $u$, then $u_k$ is relatively compact and its limit $u$ is the minimizer. Otherwise, it's exactly one of $\mu_j$, and 
\begin{equation*}\label{th3_singleton}
|u_k|^q \rightharpoondown \delta_{x_0}, \quad |\nabla u_k|^p \rightharpoondown S_{\Sigma_D}\delta_{x_0}
\end{equation*}
holds for some $x_0$. Since we split the mass between $B(0,1)$ and it's complement, $x_0$ must be on the boundary of $B(0,1)$. If $x_0$ is inside $\Sigma_D$, then $u_k$ cannot give a better constant than $S_{\mathbb{R}^n}$. 
If $x_0\in \partial \Sigma_D \cap \partial B(0,1)$, then due to the smoothness of $\partial D$, $u_k$ cannot give a better constant than $S_{\mathbb{R}_+^n}$. 
In both cases it cannot give $S_{\Sigma_D}$, which is lesser due to \eqref{eq:S_less_than_half}. Thus, the proof is concluded.
\end{proof}

\section{Symmetry breaking}
Now we return to the general case $0<\sigma\leqslant 1$. Recall from \eqref{eq:functional_Q_Hardy_Sobolev} the Hardy--Sobolev quotient 
\begin{equation*}\label{eq:functional_Q_Hardy}
Q_{\Sigma_D}^\sigma(u) = \dfrac{\|\nabla u\|_{L_p(\Sigma_D)}^p}{\|r^{(\sigma-1)}u\|_{L_q(\Sigma_D)}^p}, \qquad u\in \mathcal{D}_{p}(\Sigma_D), \quad u\neq 0,
\end{equation*}
\[
S_{\Sigma_D}^\sigma = \inf\limits_{u\in\mathcal{D}_{p}(\Sigma_D)\setminus\{0\}} Q_{\Sigma_D}^\sigma(u).
\]
For $0<\sigma<1$ the constant $S_{\Sigma_D}^\sigma$ is always attained due to Theorem \ref{Th1_Nazarov}.
For $\sigma=1$ we assume \eqref{eq:S_less_than_half} and $\partial D \in \mathcal{C}^1$ to have attainability due to Theorem \ref{theorem1}.

Consider the functional 
\[
J_{\Sigma_D}^\sigma(u) = \dfrac{1}{p}\int\limits_{\Sigma_D}|\nabla u|^p\,dx-\dfrac{1}{q}\int\limits_{\Sigma_D}|r^{(\sigma-1)}u|^q\,dx, \qquad u\in\mathcal{D}_{p}(\Sigma_D)
\]
together with the associated Nehari manifold
\[
\mathcal{N}^\sigma(\Sigma_D) = \{ u\in\mathcal{D}_{p}(\Sigma_D): u\neq 0, \ \|\nabla u\|_{L_p(\Sigma_D)}^p = \|r^{(\sigma-1)}u\|_{L_q(\Sigma_D)}^q \}.
\]
For any nontrivial $u\in\mathcal{D}_{p}(\Sigma_D)$ there exists $t_u>0$ such that $t_u u\in \mathcal{N}^\sigma(\Sigma_D)$. Then 
\[
J_{\Sigma_D}^\sigma(t_uu) = \left(\dfrac{1}{p}-\dfrac{1}{q}\right)\Big[Q_{\Sigma_D}^\sigma(u)\Big]^{\frac{q}{q-p}}.
\]
Note that 
\[
\dfrac{1}{p}-\dfrac{1}{q} = \dfrac{\sigma}{n}, \quad \frac{q}{q-p} = \dfrac{n}{\sigma p}.
\]
Therefore
\[
\inf\limits_{u\in\mathcal{N}^\sigma(\Sigma_D)} J_{\Sigma_D}^\sigma(u) = \dfrac{\sigma}{n}\Big[S_{\Sigma_D}^\sigma\Big]^{\frac{n}{\sigma p}}.
\]
Consider the problem \eqref{eq:Neumann_problem} with $\lambda=1$:
\begin{equation}\label{eq:main_problem_Hardy}
\begin{cases} 
-\Delta_p u = \dfrac{u^{q-1}}{|x|^{(1-\sigma)q}} &\text{in } \Sigma_D, \\
\dfrac{\partial u}{\partial \nu} = 0 &\text{on } \partial \Sigma_D, \\
u > 0 &\text{in } \Sigma_D.
\end{cases}
\end{equation}
This problem has a family of radial solutions $U = C a^{\frac{n}{p}-1} w(a|x|)$ for a certain $C>0$ and arbitrary $a>0$, where $w$ is defined in \eqref{eq:w_def}.
Multiplying the equation \eqref{eq:main_problem_Hardy} by $U$ and integrating we arrive at
\[
\int\limits_{\Sigma_D}|\nabla U|^p\,dx = \int\limits_{\Sigma_D}|r^{(\sigma-1)}U|^q\,dx,
\]
therefore $U\in\mathcal{N}^\sigma(\Sigma_D)$.

Since $\mathcal{N}^\sigma(\Sigma_D)$ is a manifold of codimension $1$, we can look at the second variation of $J_{\Sigma_D}^\sigma$ at the point $U$, and as long as it has at least $2$ linearly independent negative directions, there must be a better minimizer than $U$. Since the Talenti--Bliss type functions \eqref{eq:w_def} are the minimizers of the functional $Q_{\Sigma_D}^\sigma$ in the space of radial functions, this implies that the global minimizer is not radial. 

The second variation is easy to calculate (see e.g. \cite[Section 3]{Nazarov2004}):
\[
DJ_{\Sigma_D}^\sigma(U; h) = \int\limits_{\Sigma_D}|\nabla U|^{p-2}\langle\nabla U, \nabla h\rangle\,dx-\int\limits_{\Sigma_D}r^{(\sigma-1)q}|U|^{q-2}Uh\,dx,
\]
\begin{align*}
D^2J_{\Sigma_D}^\sigma(U; h, h) = &\int\limits_{\Sigma_D}|\nabla U|^{p-4}((p-2)\langle\nabla U, \nabla h\rangle^2+|\nabla U|^2|\nabla h|^2)\,dx\\
&{}-(q-1)\int\limits_{\Sigma_D}r^{(\sigma-1)q}|U|^{q-2}h^2\,dx.
\end{align*}
Therefore
\[
D^2J_{\Sigma_D}^\sigma(U; U, U) = (p-1)\int\limits_{\Sigma_D}|\nabla U|^p\,dx-(q-1)\int\limits_{\Sigma_D}|r^{(\sigma-1)}U|^q\,dx < 0
\]
due to $q > p$. Now we would like to search for a second negative direction as well. Consider the direction 
\[
h = f(|x|) g(x/|x|),
\]
where $g$ is a function on $D$. We arrive at
\begin{align*}
D^2J_{\Sigma_D}^\sigma(U; h, h) = &\int\limits_{\Sigma_D}|\nabla U|^{p-4}((p-2)\langle\nabla U, \nabla (fg)\rangle^2+|\nabla U|^2|\nabla (fg)|^2)\,dx\\
&{}-(q-1)\int\limits_{\Sigma_D}r^{(\sigma-1)q}|U|^{q-2}f^2g^2\,dx.
\end{align*}
Since $\langle \nabla U, \nabla g\rangle = 0$ and $\nabla U ||\nabla f$, we obtain
\begin{align*}
D^2J_{\Sigma_D}^\sigma(U; h, h) = &\int\limits_{\Sigma_D}|\nabla U|^{p-2}\left((p-1)g^2|\nabla f|^2 + f^2\dfrac{|\nabla_\theta g|^2}{|x|^2}\right)\,dx\\
&{} -(q-1)\int\limits_{\Sigma_D}r^{(\sigma-1)q}|U|^{q-2}f^2g^2\,dx,
\end{align*}
where $\nabla_{\theta}$ is the tangential gradient on $D$. Following the idea of \cite{Ciraolo2024} (see also \cite[Section 3]{Nazarov2004}), we assume that $g$ is the first non-constant Neumann eigenfunciton of the Beltrami--Laplace operator on $D$ and denote by $\lambda_1(D)$ the corresponding eigenvalue. 
If we also assume $\|g\|^2_{L_2(D)}=|D|$, we arrive at
\begin{align*}
D^2J_{\Sigma_D}^\sigma(U; h, h) = &\int\limits_{\Sigma_D} |\nabla U|^{p-2}\left((p-1)|\nabla f|^2 + f^2\dfrac{\lambda_1(D)}{|x|^2}\right)\,dx\\
&{} -(q-1)\int\limits_{\Sigma_D} r^{(\sigma-1)q}|U|^{q-2}f^2\,dx.
\end{align*}
In order for $h$ to be a negative direction, it is sufficient to find $f$, such that
\[
- (p-1) \mathrm{div} (|\nabla U|^{p-2} \nabla f) f - (q-1)r^{(\sigma-1)q}|U|^{q-2}f^2 < -\lambda_1(D) |\nabla U|^{p-2}\dfrac{f^2}{|x|^2}.
\]

\begin{proposition}
Let $U = U(|x|)$ be a radial solution of \eqref{eq:main_problem_Hardy}. Then function 
\[
f(x) = |x|^{\alpha} U'(|x|), \qquad \alpha = (1-\sigma)\dfrac{q}{p} < 1,
\]
solves the equation
\begin{equation}\label{eq:f_diff_eq_Hardy}
- (p-1) \mathrm{div} (|\nabla U|^{p-2} \nabla f) - (q-1)r^{(\sigma-1)q}|U|^{q-2}f = \Lambda^* |\nabla U|^{p-2}\dfrac{f}{|x|^2}
\end{equation}
in $\Sigma_D$, where $\Lambda^* = -(1-\alpha)(n-1-\alpha(p-1)) < 0$.
\end{proposition}
\begin{proof}
For the ordinary Laplacian $\Delta$ the following relation holds:
\[
(\Delta U)' = \left(U''+\dfrac{n-1}{r}U'\right)'=U'''+\dfrac{n-1}{r}U'' - \dfrac{n-1}{r^2}U'. 
\]
By the definition of $f$ we write the following derivatives:
\[
f' = r^{\alpha} U'' + \alpha r^{\alpha-1} U', \quad f'' = r^{\alpha} U''' + 2 \alpha r^{\alpha-1} U'' + \alpha(\alpha-1)r^{\alpha-2} U'.
\]
\[
\Delta f = f''+\dfrac{n-1}{r}f'=r^{\alpha} U''' + (n-1+2 \alpha) r^{\alpha-1} U'' + \alpha(n-2+\alpha)r^{\alpha-2} U'.
\]
Denote $W = |\nabla U|^{p-2}$, so we can rewrite:
\begin{align*}
-(p-1)\mathrm{div} (|\nabla U|^{p-2} \nabla f) &= -(p-1)\Big(W \Delta f + \nabla W\cdot\nabla f\Big) \\
{} &= -(p-1)\Big(W \Delta f + W' f'\Big) \\
&= -(p-1)\Big(r^{\alpha} W U''' + (n-1+2 \alpha) r^{\alpha-1} W U''\\
{} &\quad+ \alpha(n-2+\alpha)r^{\alpha-2} W U' + r^{\alpha} W' U'' + \alpha r^{\alpha-1} W' U'\Big), \\
-(q-1)r^{(\sigma-1)q}|U|^{q-2}f &= -(q-1)r^{(\sigma-1)q+\alpha}|U|^{q-2}U'\\
& = -r^\alpha (r^{(\sigma-1)q} U^{q-1})' - (1-\sigma)q r^{\alpha-1} (r^{(\sigma-1)q} U^{q-1})\\
{} &= r^\alpha (\Delta_p U)' + (1-\sigma)q r^{\alpha-1}\Delta_p U\\
&= r^\alpha (W\Delta U + \nabla W\cdot \nabla U)' \\
{} &\quad+ (1-\sigma)q r^{\alpha-1} (W\Delta U + \nabla W\cdot \nabla U) \\
{} &= r^\alpha W'U''+(n-1)r^{\alpha-1} W'U' +r^\alpha WU''' \\
 {} &\quad+(n-1)r^{\alpha-1} WU'' - (n-1) r^{\alpha-2} WU'\\
{} &\quad + r^\alpha W''U' + r^\alpha W'U'' + (1-\sigma)q r^{\alpha-1} W U''\\
{} &\quad+ (n-1)(1-\sigma)q r^{\alpha-2} W U' + (1-\sigma)q r^{\alpha-1} W'U', \\
\Lambda^* |\nabla U|^{p-2}\dfrac{f}{|x|^2} &= \Lambda^* r^{\alpha-2} W U'.
\end{align*}
Substituting these relations into \eqref{eq:f_diff_eq_Hardy} and grouping repeating terms, we get
\begin{align*}
&-(p-2)r^\alpha WU''' + \left[-(p-1) (n-1+2\alpha) + n-1 + (1-\sigma)q\right] r^{\alpha-1} W U''\\
&+ \left[ -(p-1) \alpha (n-2+\alpha) + (n-1)(1-\sigma)q -\Lambda^* - (n-1)\right] r^{\alpha-2} W U'\\
{} &- (p-3)r^\alpha W'U'' + \left[-(p-1)\alpha + n - 1 + (1-\sigma)q\right]r^{\alpha-1} W'U' + r^\alpha W''U' = 0.
\end{align*}
By the definition of $W$ we can write 
\begin{align*}
W'U' &= (p-2) W U'', \\
W''U' &= (p-2)WU'''+(p-3)W'U'',
\end{align*}
so we are left with
\begin{align*}
&\left[-p\alpha + (1-\sigma)q\right] r^{\alpha-1} W U''\\
&+ \left[ -(p-1) \alpha (n-2+\alpha) + (n-1)(1-\sigma)q -\Lambda^* - (n-1)\right] r^{\alpha-2} W U' = 0.
\end{align*}
This holds when
\begin{align*}
\alpha &= (1-\sigma)\dfrac{q}{p},\\
\Lambda^* &= -(1-\alpha)(n-1-\alpha(p-1)).
\end{align*}
Therefore, this completes the proof.
\end{proof}

Thus, we found the necessary function $f$, and it provides us a negative non-radial direction $h$ when $\lambda_1(D) < -\Lambda^*$. This gives the following statement.

\begin{theorem}\label{theorem4}
1. Let $0<\sigma<1$, and $\lambda_1(D) < (1-\alpha)(n-1-\alpha(p-1))$, where $\alpha = (1-\sigma)\dfrac{q}{p}$. Then $S_{\Sigma_D}^\sigma$ is attained at a non-radial function.

2. Let $\sigma=1$. Assume that $\partial D\in \mathcal {C}^1$, $S_{\Sigma_D}<S_{\mathbb{R}^n_+}$, and $\lambda_1(D) <n-1$. Then $S_{\Sigma_D}$ is attained at a non-radial function.
\end{theorem}

\begin{remark}
For $n\geqslant 3$, $\sigma=1$, $p=2$ the second assertion of Theorem \ref{theorem4} was proved in \cite{Ciraolo2024}.
\end{remark}

\begin{remark}
Note that for $\sigma=1$ the assumption $\lambda_1(D) <n-1$ does not depend on $p$. This restriction is sharp, since for $\lambda_1(D) \geqslant n-1$ there are examples of convex cones, which have only radial minimizers due to \cite{Lions1988}.

For the general Hardy--Sobolev case we see that the restriction on $\lambda_1(D)$ depends on $p$. It would be interesting to investigate if it is sharp in that case.
\end{remark}

\begin{remark}
Notice that under the assumptions of Theorem \ref{theorem4} the problem \eqref{eq:main_problem_Hardy} has at least two solutions.
\end{remark}

\section*{Appendix}
\label{appA}

The following proposition is a variant of Lemma 2.2 from \cite{Lions1988} (see also \cite{Lions1985}). Its proof is quite similar, but differs in the special treatment for the $\infty$ point of the compactified cone $\overline \Sigma_D$.

\begin{proposition}
\label{prop:lemma_2.2}
Let $u_k$ be a bounded sequence in $\mathcal{D}_{p}(\Sigma_D)$ with bounded supports. Then, up to a subsequence, $u_k\rightharpoondown u$, and the following weak convergences hold in the space of measures in $\overline \Sigma_D$:
\begin{align}
\label{eq:mu_representation}
|u_k|^q \rightharpoondown \mu &= |u|^q + \sum\limits_{j\in J} \mu_j \delta_{x_j}, \\
\label{eq:mu_tilde_representation}
|\nabla u_k|^p \rightharpoondown\widetilde\mu &\geqslant |\nabla u|^p + \sum\limits_{j\in J} S_{\Sigma_D} \mu_j^{p/q} \delta_{x_j},
\end{align}
where $J$ is an at most countable set, $x_j\in \overline \Sigma_D$ are distinct points, and $\mu_j > 0$.
\end{proposition}

\begin{proof}
\underline{Case 1.} Consider first the case $u\equiv 0$. For any $\varphi \in \mathcal C^\infty(\overline \Sigma_D)$ by the definition of $S_{\Sigma_D}$ we have
\begin{equation}\label{eq:Sobolev_in_comp_cone}
 \left(\,\int\limits_{\Sigma_D}|\varphi u_k|^q \, dx\right)^{1/q} \leqslant S_{\Sigma_D}^{-1/p} \left(\,\int\limits_{\Sigma_D} |\nabla(\varphi u_k)|^p \, dx\right)^{1/p}.
\end{equation}
The left-hand side integral has a limit
\[
\int\limits_{\Sigma_D}|\varphi u_k|^q \, dx \to \int\limits_{\overline \Sigma_D}|\varphi|^q \, d\mu, \quad n\to\infty.
\]
For the right-hand side in \eqref{eq:Sobolev_in_comp_cone} we write the triangle inequality
\begin{equation}\label{eq:triangle_ineq}
\Big|\|\nabla(\varphi u_k)\|_{L_p(\Sigma_D)} - \|\varphi \nabla u_k\|_{L_p(\Sigma_D)}\Big| \leqslant \| u_k\nabla\varphi\|_{L_p(\Sigma_D)}.
\end{equation}
Consider here $\varphi$ with compact support. Since the weak limit of $u_k$ is $u\equiv 0$, the last norm tends to $0$ due to the Rellich theorem, therefore passing to the limit in \eqref{eq:Sobolev_in_comp_cone} we obtain
\[
\left(\,\int\limits_{\overline \Sigma_D}|\varphi|^q \, d\mu\right)^{1/q} \leqslant S_{\Sigma_D}^{-1/p} \left(\,\int\limits_{\overline \Sigma_D} |\varphi|^p \, d\widetilde{\mu}\right)^{1/p}.
\]
By approximation this implies
\begin{equation}\label{eq:lemma_2.2_mes_abs_cont_ineq}
\mu(E)^{1/q} \leqslant S_{\Sigma_D}^{-1/p} \widetilde{\mu}(E)^{1/p}, \quad \forall E \text{ --- Borel set: } \infty\not\in E.
\end{equation}
Therefore, $\mu$ is absolutely continuous with respect to $\widetilde{\mu}$, i.e. 
\begin{equation}\label{eq:lemma_2.2_mes_abs_cont_int}
\mu(E) = \int\limits_E f \, d\widetilde{\mu}, \quad \forall E \text{ --- Borel set: } \infty\not\in E.
\end{equation}
where $f\in L^+_1(\widetilde{\mu})$,
\begin{equation}\label{eq:lemma_2.2_f_relation}
f(x) = \lim\limits_{r\to 0} \dfrac{\mu(B(x,r))}{\widetilde{\mu}(B(x,r))}, \quad \text{for } \widetilde\mu\text{-a.e. } x\in \overline \Sigma_D. 
\end{equation}
This is a consequence of the theory of symmetric
derivatives of Radon measures; see \cite[pp. 152-169]{Federer}. But from \eqref{eq:lemma_2.2_mes_abs_cont_ineq} for every $B(x,r)$, such that $\widetilde{\mu}(B(x,r))>0$ and $\widetilde\mu(\{x\})=0$ we have
\begin{equation}\label{eq:lemma_2.2_f_zero_ae}
\dfrac{\mu(B(x,r))}{\widetilde{\mu}(B(x,r))} \leqslant S_{\Sigma_D}^{-q/p} \widetilde{\mu}(B(x,r))^{(q-p)/p} \to 0, \quad r\to 0.
\end{equation}
Since $\widetilde\mu$ is a finite measure, the set
\[
J := \{x_j: \widetilde\mu(\{x_j\}) \neq 0\}
\]
is at most countable (even potentially adding $\infty$), and from \eqref{eq:lemma_2.2_mes_abs_cont_int}, \eqref{eq:lemma_2.2_f_relation}, \eqref{eq:lemma_2.2_f_zero_ae} we arrive at \eqref{eq:mu_representation}. 

\underline{Case 2.} Now let's consider $u\not\equiv 0$. The Sobolev inequality \eqref{eq:Sobolev_in_comp_cone} still holds. Denote $v_k = u_k - u$. By the Br\'{e}zis-Lieb lemma \cite{BrezisLieb}, for any $\varphi \in \mathcal C^\infty(\overline \Sigma_D)$ we have
\[
\int\limits_{\Sigma_D}|\varphi u_k|^q \, dx -\int\limits_{\Sigma_D}|\varphi v_k|^q \, dx \to \int\limits_{\Sigma_D}|\varphi u|^q \, dx, \quad n\to\infty.
\]
Applying Case 1 to $v_k$ we obtain the representation \eqref{eq:mu_representation}. Now, passing to the limit in \eqref{eq:Sobolev_in_comp_cone} and taking into account \eqref{eq:triangle_ineq} we write
\begin{equation}\label{eq:case2_ineq}
S_{\Sigma_D}^{1/p} \left(\,\int\limits_{\overline \Sigma_D}|\varphi|^q \, d\mu\right)^{1/q} \leqslant \left(\,\int\limits_{\overline \Sigma_D} |\varphi|^p \, d\widetilde{\mu}\right)^{1/p} + \left(\,\int\limits_{\Sigma_D} |\nabla\varphi|^p|u|^p \, dx\right)^{1/p}.
\end{equation}
For $x_j\neq\infty$ choose $\psi(x)\in \mathcal{C}^\infty_0(B(0,1))$ such that $0\leqslant\psi\leqslant 1$, $\psi(0) = 1$, and then substitute $\varphi(x) = \psi((x-x_j)/\varepsilon)$ into \eqref{eq:case2_ineq}. We obtain 
\[
\mu_j = \int\limits_{\overline \Sigma_D} 1 \, d(\mu_j\delta_{x_j}) \leqslant \int\limits_{\overline \Sigma_D}|\varphi|^q \, d\mu,
\]
\[
\int\limits_{\overline \Sigma_D} |\varphi|^p \, d\widetilde{\mu} \leqslant \widetilde{\mu}(B(x_j, \varepsilon)),
\]
\[
\int\limits_{\Sigma_D} |\nabla\varphi|^p|u|^p \, dx = \int\limits_{B(x_j, \varepsilon)} |\nabla\varphi|^p|u|^p \, dx.
\]
Hereinafter, in the integrals over $B(x_j, \varepsilon)$ we omit the integral outside the cone. Therefore,
\[
S_{\Sigma_D}^{1/p} \mu_j^{1/q} \leqslant (\widetilde{\mu}(B(x_j, \varepsilon))^{1/p} + \left(\,\int\limits_{B(x_j, \varepsilon)} |\nabla\varphi|^p|u|^p \, dx\right)^{1/p}.
\]
By the H\"{o}lder inequality we estimate
\begin{align*}
\int\limits_{B(x_j, \varepsilon)} |\nabla\varphi|^p|u|^p \, dx &\leqslant \left(\,\int\limits_{B(x_j, \varepsilon)} |u|^q \, dx\right)^{p/q}\left(\,\int\limits_{B(x_j, \varepsilon)} |\nabla\varphi|^n \, dx\right)^{p/n}\\
{} &\leqslant \left(\,\int\limits_{B(x_j, \varepsilon)} |u|^q \, dx\right)^{p/q}\left(\,\int\limits_{B(0, 1)} |\nabla\psi|^n \, dx\right)^{p/n}.
\end{align*}
Therefore, 
\[
S_{\Sigma_D}^{1/p} \mu_j^{1/q} \leqslant (\widetilde{\mu}(B(x_j, \varepsilon))^{1/p} + C\left(\,\int\limits_{B(x_j, \varepsilon)} |u|^q \, dx\right)^{1/q}.
\]
As $\varepsilon\to0$, this implies that $\widetilde{\mu}(\{x_j\}) > 0$ and
\[
\widetilde{\mu} \geqslant S_{\Sigma_D} \mu_j^{p/q} \delta_{x_j}, \quad \forall j\in J: x_j\neq\infty.
\]
Now, for the case $x_j = \infty$, consider $\psi(x)\in \mathcal{C}^\infty(\overline \Sigma_D)$ such that $0\leqslant\psi\leqslant 1$, $\psi(x) = 1$ for $|x|>2$, $\mathrm{supp}\,\psi = \overline \Sigma_D \setminus \overline{B}(0,1)$, and then substitute $\varphi(x) = \psi(x/R)$ into \eqref{eq:case2_ineq}. We similarly obtain
\[
S_{\Sigma_D}^{1/p} \mu_j^{1/q} \leqslant (\widetilde{\mu}(\overline \Sigma_D \setminus \overline{B}(0,R))^{1/p} + C\left(\,\int\limits_{\Sigma_D \setminus \overline{B}(0,R)} |u|^q \, dx\right)^{1/q}.
\]
This implies
\[
\widetilde{\mu} \geqslant S_{\Sigma_D} \mu_j^{p/q} \delta_{x_j}, \quad x_j = \infty.
\]
By weak convergence we also have $\widetilde{\mu}\geqslant|\nabla u|^p$, and since $|\nabla u|^p$ and $\delta_{x_j}$ are orthogonal, we obtain \eqref{eq:mu_tilde_representation}.
\end{proof}

\subsection*{Acknowledgements}

This research was supported by the Theoretical Physics and Mathematics Advancement Foundation ``BASIS''.


\begin{thebibliography}{0}

\bibitem{Aubin1976}
Aubin, T., Problemes isop\'{e}rim\'{e}triques et espaces de Sobolev. Journal of Differential Geometry {\bf 11} (1976), N4, pp.573-598.

\bibitem{BrezisLieb}
Br\'{e}zis, H. and Lieb, E., A relation between pointwise convergence of functions and convergence of functionals. Proceedings of the American Mathematical Society {\bf 88} (1983), N3, pp.486-490.

\bibitem{Ciraolo2020}
Ciraolo, G., Figalli, A. and Roncoroni, A., Symmetry results for critical anisotropic $p$-Laplacian equations in convex cones. Geometric and Functional Analysis {\bf 30} (2020), N3, pp.770-803.

\bibitem{Ciraolo2024}
Ciraolo, G., Pacella, F. and Polvara, C.~C., Symmetry breaking and instability for semilinear elliptic equations in spherical sectors and cones. Journal de Mathématiques Pures et Appliqu\'{e}es {\bf 187} (2024), pp.138-170.

\bibitem{ClappPacella2019}
Clapp, M. and Pacella, F., Existence of nonradial positive and nodal solutions to a critical Neumann problem in a cone. Mathematics in Engineering {\bf 3} (2021), N3, pp.1-15.

\bibitem{Demyanov2006}
Demyanov, A.~V. and Nazarov, A.~I. On the existence of an extremal function in Sobolev embedding theorems with limit exponent. Algebra \& Analysis {\bf 17} 
(2005), N5, pp.105-140 (Russian); English transl.: 
St.~Petersburg Mathematical Journal 
{\bf 17} (2006), N5, pp.108-142.

\bibitem{Egnell1992}
Egnell, H., Positive solutions of semilinear equations in cones. Transactions of the American Mathematical Society {\bf 330} (1992), N1, pp.191-201.

\bibitem{Federer}
Federer, H., Geometric measure theory. Springer (2014).

\bibitem{Ghoussoub2000}
Ghoussoub, N. and Yuan, C., Multiple solutions for quasi-linear PDEs involving the critical Sobolev and Hardy exponents. Transactions of the American Mathematical Society {\bf 352} (2000), N12, pp.5703-5743.

\bibitem{Lieb1983}
Lieb, E.~H., Sharp constants in the Hardy--Littlewood--Sobolev and related inequalities. In Inequalities: Selecta of Elliott H. Lieb (1983), pp. 529-554. Berlin, Heidelberg: Springer Berlin Heidelberg.

\bibitem{Lions1984}
Lions, P.~L., The concentration-compactness principle in the Calculus of Variations. The locally compact case, part 1. Annales de l'Institut Henri Poincar\'{e} C, Analyse Non Lin\'{e}aire {\bf 1} (1984), N2, pp.109-145.

\bibitem{Lions1985}
Lions, P.~L., The concentration-compactness principle in the calculus of variations. The limit case, part 1. Revista Matem\'{a}tica Iberoamericana {\bf 1} (1985), N1, pp.145-201.

\bibitem{Lions1988}
Lions, P.~L., Pacella, F. and Tricarico, M., Best constants in Sobolev inequalities for functions vanishing on some part of the boundary and related questions. Indiana University Mathematics Journal {\bf 37} (1988), N2, pp.301-324.

\bibitem{Nazarov2004}
Nazarov, A.~I., On solutions of the Dirichlet problem for an equation involving
the $p$-Laplacian in a spherical layer. Proc. St.Petersburg Math. Soc. {\bf 10} (2004), pp.33-62 (Russian); English transl.: AMS 
Translations Series 2. {\bf 214} (2005), pp.29-57.

\bibitem{Nazarov2006}
Nazarov, A.~I., Hardy-Sobolev inequalities in a cone. Probl. Mat. Anal. {\bf 31} (2005), pp.39-46 (Russian); English transl.: Journal of Mathematical Sciences {\bf 132} 
(2006), N4, pp.419-427.

\bibitem{Talenti1976}
Talenti, G., Best constant in Sobolev inequality. Annali di Matematica pura ed Applicata {\bf 110} (1976), pp.353-372.


\end{thebibliography}
\end{document}